\documentclass[reqno]{amsart}

\theoremstyle{definition}

\usepackage[T1]{fontenc}
\usepackage[utf8]{inputenc}
\usepackage{lmodern}
\usepackage{textcomp}
\usepackage{microtype}
\usepackage{tikz}
\usepackage{mathtools, bm, nccmath}

\usepackage[english]{babel}
\usepackage{csquotes}

\usepackage[backend=biber,
			style=numeric,
			isbn=false,
			doi=false
			]{biblatex}
\usepackage{tabularx}

\usepackage{graphicx}
\usepackage{setspace}

\usepackage{amsmath}
\usepackage{amsfonts}
\usepackage{amssymb}
\usepackage{amsthm}
\numberwithin{equation}{section}

\newcommand*{\N}{\mathbb{N}}
\newcommand*{\Z}{\mathbb{Z}}

\newcommand*{\R}{\mathbb{R}}

\newcommand*{\ta}{\tilde{a}}

\usepackage{thmtools}

\declaretheorem[
	name=Theorem,
	numberwithin=section
	]{thm}

\declaretheorem[
	name=Corollary,
	sibling=thm,
	]{cor}

\declaretheorem[
	name=Definition,
	style=definition,
	]{defin}

\declaretheorem[
	name=Remark,
	style=remark,
	numbered=no
	]{rem}

\usepackage[pdftex,pdfpagelabels]{hyperref}

\numberwithin{equation}{section}

\makeatletter
\@namedef{subjclassname@2020}{\textup{2020} Mathematics Subject Classification}
\makeatother

\calclayout

\begin{document}
\title[Dispersion of dilated lacunary sequences]
{THE DISPERSION OF DILATED LACUNARY SEQUENCES, WITH APPLICATIONS IN MULTIPLICATIVE DIOPHANTINE APPROXIMATION}

\author[E. Stefanescu]{Eduard Stefanescu}
\address{Institut f\"ur Analysis und Zahlentheorie, TU Graz, Steyrergasse 30, 8010 Graz, Austria}
\email{\href{mailto:eduard.stefanescu@tugraz.at}{eduard.stefanescu@tugraz.at}}

\subjclass[2020]{11J83; 11J70; 42A16; 28A78.}
\keywords{Number Theory, Dispersion, Littlewood Conjecture}

\begin{abstract}
Let $(a_n)_{n \in \mathbb{N}}$ be a Hadamard lacunary sequence. We give upper bounds for the maximal gap of the set of dilates $\{a_n \alpha\}_{n \leq N}$ modulo 1, in terms of $N$. For any lacunary sequence $(a_n)_{n \in \mathbb{N}}$ we prove the existence of a dilation factor $\alpha$ such that the maximal gap is of order at most $(\log N)/N$, and we prove that for Lebesgue almost all $\alpha$ the maximal gap is of order at most $(\log N)^{2+\varepsilon}/N$. The metric result is generalized to other measures satisfying a certain Fourier decay assumption. Both upper bounds are optimal up to a factor of logarithmic order, and the latter result improves a recent result of Chow and Technau.  Finally, we show that our result implies an improved upper bound in the inhomogeneous version of Littlewood's problem in multiplicative Diophantine approximation.
\end{abstract}

\maketitle

\section{Introduction}

In this paper we investigate the dispersion, that is the largest gap, in sets which arise as dilates of a lacunary sequence modulo one. Throughout this paper, $(a_n)_{n \in \mathbb{N}}$ denotes an increasing sequence of integers satisfying the Hadamard gap condition with some growth parameter $r>1$, that is, $a_{n} \geq r a_{n-1}$ for $n \geq 2$, and $\{a_n\}_{n  \leq N}$ denotes a finite initial segment of such a sequence.  For $\alpha \in [0,1]$ we write $G(\{\alpha a_n\}_{n\le N})$ for the \textit{maximal gap} between neighbouring fractional parts of dilates $\{\alpha a_n\}_{n\le N}$ on the unit torus. More precisely, let $\{\alpha \hat{a}_{n}\}_{n\le N}$ be the order statistics of $\{\alpha a_n\}_{n\le N}$, understood to be taken modulo one so that all numbers are elements of the unit torus. Then 
$$
0 \leq \{\alpha \hat{a}_1\} \le \{\alpha \hat{a}_2\} \le...\le \{\alpha \hat{a}_N\} < 1,
$$
and the largest gap is defined as 
$$G(\{\alpha a_n\}_{n\le N}):=\max\limits_{{2\le k\le N+1}}(\{\alpha \hat{a}_n\}- \{\alpha \hat{a}_{n-1})\},$$
where $\alpha \hat{a}_{N+1}:=1+\alpha \hat{a}_1$. For a given lacunary sequence $(a_n)_{n\in\N}$, we will show that we can always find an $\alpha$ such that 
\begin{equation} \label{eq_gap_1}
G(\{\alpha a_n\}_{n\le N})\ll_r\frac{\log(N)}{N}
\end{equation}
as $N \to \infty$. We also study the metric problem and show that for every $\varepsilon > 0
$, for Lebesgue almost all dilation factors $\alpha\in[0,1]$,
\begin{equation} \label{eq_gap_2}
G(\{\alpha a_n\}_{n\le N}) \ll  \frac{\log(N)^{2 + \varepsilon}}{N}
\end{equation}
as $N \to \infty$. A similar result with exponent 3 instead of 2 for the logarithmic term has been obtained recently by Chow and Technau \cite{chow2023dispersion}. We obtain an analogous result for more general measures, which allows for an application in the context of Littlewood's conjecture in multiplicative Diophantine approximation, where we can show that for $\varepsilon>0$ and $\eta,\zeta\in\R$ we have
$$n||\alpha n-\eta||\cdot||\beta n-\zeta||\leq \frac{\log(\log(n))^{2+\varepsilon}}{\log(n)}$$
for ``generic'' $\alpha$ and $ \beta$, thereby again quantitatively improving a result of Chow and Technau.  Here and throughout the paper, $\| \cdot \|$ denotes the distance to the nearest integer. More details on this application in Diophantine approximation (in particular the meaning of the word ``generic'' in the previous sentence) will be described in Section \ref{sec:little} at the end of this paper. \\

Dilates of lacunary sequences and lacunary function series have a long and rich history in analysis and number theory. It is a classical observation that lacunary trigonometric series (or, more generally, lacunary series of dilates of periodic functions) mimic the typical behavior of sequences of independent random variables; this is the core message of classical papers such as those by Salem and Zygmund \cite{sz1, sz2}, Kac \cite{kac}  and Gapo\v{s}kin \cite{gapo}, and is described in detail in the survey paper \cite{abt}. This analogy between lacunary systems and independent random variables also underpins the contents of the present paper, even if it is not explicitly emphasized. Another source for this paper is the study of gap properties of dilates of (not necessarily lacunary) integer sequences modulo one. This topic is discussed in its own right, for example, by Konyagin, Ruzsa and Schlag \cite{krs} and Rudnick \cite{rudnick}, and can be traced back to work of Kronecker, Sierpi\'nski, Weyl, and many others. One particularly interesting connection is with the study of correlation functions and neighbor spacings in the context of the famous Berry--Tabor conjecture from quantum chaology (see \cite{mark} for some background); in that setup, dilates of lacunary sequences have been proven to follow the random (``Poissonian) model for almost all dilation parameters, see for example Rudnick and Zaharescu \cite{rz}, Chaubey and Yesha \cite{cy} and Aistleitner, Baker, Technau and Yesha \cite{abty}. However, it should be noted that in the Berry--Tabor conjecture one is interested in the overall distributional structure of the system of gaps, while in this paper we are only interested in bounding the extremal (maximal) gaps.\\

Trivially, for every sequence $(a_n)_{n \in \mathbb{N}}$ and for every $\alpha$ one has $G(\{\alpha a_n\}_{n\le N}) \geq 1/N$. Thus the upper bound from \eqref{eq_gap_1} is optimal up to a logarithmic factor. The metric result in equation \eqref{eq_gap_2} should of course be compared with the corresponding result for the random case, where Devroye \cite{devroye} proved that for a sequence $(X_n)_{n \in \mathbb{N}}$ of i.i.d.\ random variables having uniform distribution on $[0,1]$, $G(\{X_n\}_{n\le N})$ is of typical order roughly $(\log N) / N$ (the result obtained by Devroye is much more precise, see \cite{devroye} for details). In accordance with these results for the random case, it is tempting to conjecture that for a lacunary sequence $(a_n)_{n \in \mathbb{N}}$ one always has 
$$
G(\{\alpha a_n\}_{n\le N}) \gg \frac{\log(N)^{1 -\varepsilon}}{N},
$$
as well as 
$$
G(\{\alpha a_n\}_{n\le N}) \ll \frac{\log(N)^{1+\varepsilon}}{N}
$$
for every $\varepsilon > 0$ and almost all $\alpha$. If this conjecture is indeed correct, then our upper bound \eqref{eq_gap_2} is one logarithmic factor away from optimality. The problem of the exact asymptotic order of the maximal gap of dilated lacunary sequences mod 1 remains open, in both setups that we study in this paper. We consider this to be a very interesting problem, since it pertains to the question how robust the analogy between i.i.d.\ random systems and dilated lacunary systems is in the specific context of very rare events (see in this context also \cite{agkpr}). We recall in passing that a paper of Peres and Schlag \cite{ps} also contains a (still unsolved) problem on gaps of dilates of lacunary sequences mod 1; their problem, however, is of a rather different nature than ours.

\section{Main Results} In this section we give the precise statement of all main results of this paper. The notation is explained later at the end of this section.\\

\begin{defin}[Lacunary sequence]
    A sequence $(a_n)_{n\in\N}\subset \R$ is called a \textit{lacunary sequence}, if for a fixed \textit{growth factor} $r>1$:
    $$a_{n+1}\ge ra_{n}$$
    for every $n\in\N.$
    Let $\mathfrak{N}\subseteq\N$ be finite. We say $\{a_n\}_{n\in\mathfrak{N}}:=\{a_n\in(a_n)_{n\in\N}|n\in\mathfrak{N}\}$ is a set of \textit{lacunary type} with growth factor $r>1$ if $(a_n)_{n\in\N}$ is a lacunary sequence with the same growth factor.
\end{defin}
The first main result gives an upper bound to the maximal gap of dilated sets of lacunary type, where the dilation parameter can depend on the cardinality of the set.
\begin{thm}\label{main}
 Let $\{a_n\}_{n\le N}$ be a set of lacunary type with growth factor $r>1$. Then
 for all $N\in\N$ with $N\ge N_0(r)$, there exists an $\alpha\in\R$, such that the maximal gap of the set $\{\alpha a_n\}_{n\in N}$ is at most $\log(N)/N$ up to constants, i.e.
 $$G(\{\alpha a_n\}_{n\le N})\ll_r\frac{\log(N)}{N}.$$
\end{thm}
As the proof of Theorem \ref{main} will show, the value of $\alpha$ is not unique, and every sufficiently large interval $I$ contains a value of $\alpha$ for which the conclusion of the theorem holds.\\ 

The next main result improves the first one by allowing the same dilation parameter for all $N$.  
\begin{thm}\label{main2}
 Let $(a_n)_{n\in\N}$ be a lacunary sequence with growth factor $r>1$. Then there exists an $\alpha\in\R$ such that the maximal gap of the first $N$ elements of $(\alpha a_n)_{n\in\N}$ is at most $\log(N)/N$ for all sufficiently large $N$, i.e.
 $$G(\{\alpha a_n\}_{n\le N})\ll_r\frac{\log(N)}{N}.$$
\end{thm}

In the third main result we study the metric case, that is, we give an upper bound for the maximal gap which holds for almost all dilation parameters $\alpha$ in the sense of Lebesgue measure. 

\begin{thm}\label{main3}
   Let $\varepsilon>0$ and $(a_n)_{n\in\N}$ be a lacunary sequence with growth factor $r>1$.
 Then for almost all $\alpha$ the maximal gap of the set $\{\alpha a_n\}_{n\le N}$ is at most $\log(N)^{2+\varepsilon}/N$ for all sufficiently large $N$, i.e. for Lebesgue-almost all $\alpha\in[0,1]$ there exists an $N_0(\alpha)$, such that
\begin{equation}\label{introeq}
       G(\{\alpha  a_n\}_{n\le N})\le  \frac{\log(N)^{2+\varepsilon}}{N}, 
    \end{equation}
    for all $N\in\N$, with $N\ge N_0(\alpha)$.
\end{thm}

This theorem can be extended to more general measures, subject to a condition on the decay rate of the Fourier transform of the measure. This result will be relevant for our application in the context of multiplicative Diophantine approximation, where $\alpha$ will have to be chosen from a Lebesgue null set.

\begin{thm}\label{main35}
   Let $\varepsilon>0$, and let $(a_n)_{n\in\N}$ be a lacunary sequence with growth factor $r>1$. Let a measure $\mu\in\mathcal{M}[0,1]$ be given and assume that its Fourier Transform decays as $\left|(\mathcal{F}\mu)(x)\right|\ll(1+|x|)^{-\Upsilon}$ for some $\Upsilon>0$.
 Then the maximal gap of the set $\{\alpha a_n\}_{n\le N}$ is at most $\log(N)^{2+\varepsilon}/N$ for all sufficiently large $N$, i.e. for $\mu$-almost all $\alpha\in[0,1]$, there exists an $N_0(\alpha)$, such that
\begin{equation}\label{introeq2}
       G(\{\alpha  a_n\}_{n\le N})\le \frac{\log(N)^{2+\varepsilon}}{N}, 
    \end{equation}
    for all $N\in\N$, with $N\ge N_0(\alpha)$.
\end{thm}

As a consequence of Theorem \ref{main35}, we obtain a quantitative improvement in an inhomogeneous version of the Littlewood conjecture in multiplicative Diophantine approximation. This is inspired by the recent work of Chow and Technau \cite{chow2023dispersion}. For the statement of the theorem, we define $\mathcal{L}:=\{x \in\R:\Lambda(\beta)<\infty\}$, where $\Lambda(\beta)=\sup\left\{\log(p_k(\beta))/k:k\ge 1\right\}$, with $p_k$ being the continuants of $\beta$ and  $\textnormal{\textbf{Bad}}:=\{\alpha\in\R:\inf\{n||n\alpha||>0\}$. We will provide more context for this result in Section \ref{sec:little}.

\begin{thm}\label{little}
Let $\varepsilon>0$, $\beta\in \mathcal{L}$ and $\zeta\in\R$.
Then there exists a set $E\subseteq \textnormal{\textbf{Bad}}$ with Hausdorff-dimension $\dim_H(E)=1$, such that for $\alpha\in E$ and $\eta\in\R$ there are infinitely many solutions $n$ to 
$$n||\alpha n-\eta||\cdot||\beta n-\zeta||\le \frac{\log(\log(n))^{2+\varepsilon}}{\log(n)}.$$ 
\end{thm}

Notation: Throughout this paper, we write $\Z$ for the set of integers, $\N:=\Z_{>0}$ for the set of positive integers, $\R$ for the real numbers and $\R^+$ the positive real numbers. We denote the zero-vector as $\mathbf{0}$. The Vinogradov symbol ''$\ll$'' means 'less than or equal to,' up to constants, $\simeq$ means equal up to constants. The distance to the nearest integer is denoted by $||.||:=\min_{n\in\Z}|.-n|$. For $a,b\in\R$, $\lfloor a\rfloor$ and $\lceil b\rceil$ are the next smallest and the next largest integer of $a$ and $b$ respectively, i.e. $\lfloor a\rfloor:=\max_{n\in\N}(n:n\le a)$, $\lceil b\rceil=\min_{n\in\N}(n:n\ge b)$ respectively.
We write $\{a\}$ for the fractional part of $a$, i.e. if $a>0:$ $\{a\}=a-\lfloor a\rfloor$; if $a<0:$ $\{a\}=a-\lceil a\rceil$.  When $I$ is an interval, we  denote its  length by $|I|$.

\section{Small dispersion for a suitable dilation parameter. Proof of Theorem \ref{main}.}

We use Tur\'an's localized and quantitative version of Kronecker's Theorem, see Chapter 4 in \cite{GONEK2016506}:
\begin{thm}\label{turan}[Tur\'an's localized and quantitative Kronecker Theorem]\\
For $K>0$, let $x_1,x_2,...,x_K$ and $a_1,a_2,...,a_K$ be given real numbers, let $\varepsilon_1,\varepsilon_2,...,\varepsilon_K$ be positive real numbers with $0<\varepsilon_n<1/2$ for all $n$. Let $$M_n:=\left\lceil\frac{1}{\varepsilon_n}\log\left(\frac{K}{\varepsilon_n}\right)\right\rceil,\quad\delta:=\min\left|\sum_{j=1}^{K}m_ja_j\right|,$$
where the minimum is taken over all lattice points $\mathbf{m}:=(m_1,m_2,...,m_K)\in\Z^K\setminus\{\mathbf{0}\}$, with $|m_j|\le M_n$.
If $\delta>0$, then in any interval $I$ of length at least $4/\delta$, there is a real number $\alpha$ such that    
\begin{equation}\label{turaneq}
  ||\alpha a_n-x_n||\le\varepsilon_n.  
\end{equation}
\end{thm}

Now we explain how Theorem \ref{main} follows from Tur\'an's Theorem \ref{turan}.

\begin{proof}[Proof of Theorem \ref{main}.]
    We first define a subset of our given set of lacunary type, which satisfies the assumptions of Tur\'an's Theorem \ref{turan}, which then yields a distance between equidistant points on the torus and the points of the dilated subset of at most $\log(N)/N$ up to constants, which concludes the proof.\\
 \textbf{Step 1:} 
Let $l:=\max\{1,\lceil\frac{1}{\log (r)}\rceil+1/2\}$, so that $r^l > e$.
 We define a subset $\{\tilde{a}_n\}_{n\le K}$ such that $\tilde{a}_n=a_{ln\lfloor\log(N)\rfloor}$, and $K=\lfloor N/(l\log(N)) \rfloor$, so that \begin{equation}\label{ta1}
    \tilde{a}_{n+1}\ge r^{l\lfloor\log(N)\rfloor}\tilde{a}_n\ge N^{\xi}\tilde{a_n},
\end{equation}
 for $\xi:=l\log(r)>1$, and accordingly for $m\le n$
 \begin{equation}\label{ta2}
     \tilde{a}_{m}\le N^{-\xi(n-m)}\tilde{a}_n.
 \end{equation}
Claim: The set $\{\tilde{a}_n\}_{n\le K}$ satisfies  
\begin{equation}\label{delta0}
        \delta_M:=\min_{(m_1,...,m_K)\in(\Z^K\setminus\{\mathbf{0}\}\cap[-M,M]^K)}|\sum_{j=1}^Km_j\tilde{a}_j|>0
    \end{equation}
    for $N$ sufficiently large, where $M:=\left\lceil\frac{1}{\varepsilon}\log\left(\frac{K}{\varepsilon}\right)\right\rceil$ and $\varepsilon:=\frac{l\log(N)}{2N}$.
    Observe that for every $M<\tilde{M}$ we have $\delta_M\ge\delta_{\tilde{M}}$.
This justifies estimating $M$ by something larger:
\begin{equation}\label{Mest1}
        M=\left\lceil\frac{1}{\varepsilon}\log\left(\frac{K}{\varepsilon}\right)\right\rceil\le
        \frac{2N}{l\log(N)}\log\left(\frac{\left\lfloor\frac{N}{l\log(N)}\right\rfloor}{\frac{l\log(N)}{2N}}\right)+1\le 4l^{-1}N
\end{equation}
for $N$ sufficiently large. Next, for given $(m_1,...,m_K)\in(\Z^K\setminus\{\mathbf{0}\}\cap[-4Nl^{-1},4Nl^{-1}]^K)$ we define $K_0$ as the largest index such that $m_j\ne0$, i.e. $m_{K_0}\ne0,m_{K_0+1},...,m_{K}=0$. This $K_0$ exists since $\mathbf{m}\ne\mathbf{0}$ by assumption. 
The idea is to take the last nonzero summand $\tilde{a}_{K_0}$, multiply it by the smallest possible $m_{K_0}$, i.e. $m_{K_0}=1$, and multiply all the other summands by the largest possible $m_j$, i.e. $m_j=4l^{-1}N$. If the $K_0$-th summand still dominates all the others, then the complete sum cannot be zero.
By equation \eqref{ta1}, \eqref{ta2} and the geometric series we have
\begin{align}\label{estimate}
\begin{split}
   &\ \ \ \sum_{j=1}^{K_0-1}m_j\tilde{a}_j\le4l^{-1}N\sum_{j=1}^{K_0-1} \tilde{a}_j\stackrel{\eqref{ta2}}{\le}4l^{-1}N\sum_{j=1}^{K_0-1}N^{-\xi(K_0-j)}\ta_{K_0} \\
   &\le4l^{-1}N^{-\xi K_0+1}\frac{N^{\xi(K_0-1)}-1}{N^{\xi}-1}N^{\xi}\ta_{K_0}\le 8l^{-1}N^{1-\xi}\ta_{K_0}\le\tilde{a}_{K_0},
    \end{split}
\end{align} 
for all $|m_j|\le4l^{-1}N$ and $N$ sufficiently large.
This implies $\delta_{4l^{-1}N}>0$, hence $\delta_M>0$.\\
\textbf{Step 2:}
Let $x_1,...,x_K$ be equidistant points on the torus, $\varepsilon_1=...=\varepsilon_K=\varepsilon$ and hence $M_1=...=M_K=M$. By the first step we can apply Tur\'an's Theorem \ref{turan} to $\{\tilde{a}_n\}_{n\le K}$ and get existence of an $\alpha\in\R$ in every interval of length larger than $4/\delta_M$, such that 
$$||\alpha \tilde{a}_n-x_n||\ll\varepsilon,$$
for all $n=1,...,K$.
 Hence the maximal gap of the set $\{\alpha \tilde{a}_n\}_{n\le K}$ is at most $\log(N)/N$ up to constants.
   Since adding more points can only decrease the maximal gap-size, the same is true for every superset of $ \left\{\alpha \tilde{a}_n \right\}_{n\le K} $, and so the maximal gap of $\{\alpha a\}_{n\le N}$ is also at most $\log(N)/N$ up to constants.\\
   Furthermore, by equation \eqref{estimate} we have 
   \begin{equation}\label{delta5}
     \delta>\left|\tilde{a}_{K_0}\left(1-8l^{-1}\frac{N}{N^{\xi}}\right)\right|. 
  \end{equation}
  Taking the smallest possible $K_0$, i.e. $K_0=1$, implies that there exists an $\alpha$ in every interval $I$ of length
  \begin{equation}
      |I|>\frac{4}{\tilde{a}_1}=\frac{4}{N^{\xi}}
  \end{equation}
  for $N$ sufficiently large.
\end{proof}
\begin{cor}\label{translated}
    Let $(a_n)_{n\in\N}$  be a lacunary sequence with growth factor $r>1$.
    Then for all $N\in\N$ sufficiently large, there exists an $\alpha\in\R$ in every interval $I$ with
    \begin{equation}\label{bound11}
      |I|\ge \frac{4}{a_{N}},
    \end{equation}
    such that the maximal gap of the set $\{\alpha a_n\}_{N\le n\le 2N}$ is at most $\log(N)/N$ up to constants, i.e.
    \begin{equation}\label{gap100}
        G(\{a_n\}_{N<n\le 2N})\ll_r\frac{\log(N)}{N}.
    \end{equation}
\end{cor}
\begin{proof}
First consider a subset $\{\tilde{a}_n\}_{K<n<2K}$ with $\tilde{a}_n=a_{nl\log(2N)}$, $K=N/l\log(N)$ and $l$ as before.
    We only need to adapt equation \eqref{estimate} to the new indices:
    \begin{align}\label{estimate2}
\begin{split}
   &\ \ \ \ \sum_{j=K+1}^{K_0-1}m_j\tilde{a}_j\stackrel{\eqref{ta2}}{\le}4l^{-1}N\sum_{j=K+1}^{K_0-1}N^{-\xi(K_0-j)}\ta_{K_0} \\
   &\le4l^{-1}N^{-\xi K_0+1}\frac{N^{\xi(K_0-1)}-N^K}{N^{\xi}-1}N^{\xi}\ta_{K_0}\le 8l^{-1}N^{1-\xi}\ta_{K_0}\le\tilde{a}_{K_0},
    \end{split}
\end{align} 
    where $K<K_0\le2K$ is defined as before.
    The smallest possible $K_0$ is $K+1$, which yields the new lower bound \eqref{bound11} for the length of the interval in which a suitable dilation variable $\alpha$ exists. 
\end{proof}

\begin{cor}
     Let $\vartheta>1$ and $\{a_n\}_{n\le N}$ be a set such that:
    $$a_n\ge N^{\vartheta}a_{n-1},$$
    for $2\le n\le N$. Then there exists an $\alpha\in\R$ such that:
     $$G(\{\alpha a_n\}_{n\le N})\ll\frac{1}{N}.$$
\end{cor}
\begin{proof}
    We repeat the same procedure as before. In particular we adapt equations \eqref{Mest1} and \eqref{estimate} (where we set $l=1$). For $\varepsilon=1/N$ it follows that
    $$M=4N\log\left(\frac{\sqrt{2}N}{\log(N)}\right)\le 4N\log(N),$$
    for $N$ sufficiently large.
Similar to equation \eqref{estimate} we have 
$$\sum_{j=1}^{K_0-1}m_ja_j\le4N\log(N)\sum_{j=1}^{K_0-1} a_j\le 8l^{-1}N^{1-\xi}\log(N)a_{K_0}\le a_{K_0},$$
for $N$ sufficiently large, where the constant $K_0$ is defined as before. We use Tur\'an's localized and quantitative Kronecker Theorem \ref{turan} to get the result.
\end{proof}

\section{Small dispersion for a suitable fixed dilation parameter. Proof of Theorem \ref{main2}.}

In the previous section we showed that for a set of lacunary type $\{a_n\}_{n\le N}$ with fixed cardinality we can always find a number $\alpha\in\R$ such that the dilated set $\{\alpha a_n\}_{n\le N}$ has a maximal gap of size at most $\log(N)/N$, up to constants. Note that in this statement, $\alpha = \alpha(N)$ is allowed to depend on the cardinality $N$ of the set. In this section we show that for an infinite lacunary sequence $(a_n)_{n\in\N}$ we can find a number $\alpha\in\R$ such that the first $N$ elements of the dilated set $\{\alpha a_n\}_{n\le N}$ have a maximal gap of size at most $\log(N)/N$, up to  constants, for all sufficiently large $N$; here $\alpha$ is not allowed to depend on the cardinality of the set anymore. We will use Theorem \ref{main} together with a nested interval argument to prove this.

\begin{proof}[Proof of Theorem \ref{main2}.]
We first briefly explain the idea of the proof. We consider a lacunary sequence and look at the first $N_k:=4^k$ elements, where we find an $\alpha_1$ such that the desired upper bound holds, provided that $k$ is large enough. Then we find a full interval $\tilde{I}_1$ around $\alpha_1$, such that every real number in $\tilde{I}_1$ satisfies the desired result. Next, we consider the elements between $N_{k+1}$ and $2N_{k+1}$ of the lacunary sequence, where we have that the dilation variable $\alpha_2$, that satisfy the result, can be chosen from an interval $I_1\subseteq \tilde{I}_1$. This means $\alpha_2$ satisfies the result for the first $N_k$ and $N_{k+1}$ elements of the lacunary sequence. By repeating this we find infinitely many numbers in $\N$ such that the same $\alpha$ satisfies the condition. An interpolation argument yields the result for all $N\in\N$, (up to finitely many).\\   
    We fix an $k$ sufficiently large such that there exists an $\alpha_1$ with 
    \begin{equation}\label{gap1}
       G(\{\alpha_1  a_n\}_{n\le N_k})\ll \frac{\log(N_k)}{N_k}, 
    \end{equation}
    by the first main Theorem \ref{main}. This means that for every $a_i\in\{a_n\}_{n\le N_k}$ there exists a $x_j$ in a set of $\lfloor N_k/\log(N_k)\rfloor$ equidistant points on the unit torus, such that the distance is less than $\log(N_k)/N_k$ up to constants. 
In particular for $\mathfrak{a}_1:=\alpha_1-\tau_1$, $\mathfrak{a}_2:=\alpha_1+\tau_1$, $\tau_1=\frac{\log(N_k)}{N_k\cdot a_{N_k}}$,
   \begin{align}\label{interval}
   \begin{split}
    &\ \ \ \ ||\mathfrak{a}_m a_i-x_j||=||(\alpha_1\pm \tau_1) a_n-x_n||\\
&\le||\alpha_1a_n-x_n||+||\tau_1 a_n||\stackrel{\textnormal{\ref{main}}}{\ll}\frac{\log(N_k)}{N_k}+\tau_1 a_{N_k}\\
&=\frac{\log(N_k)}{N_k}+\frac{\log(N_k)}{N_ka_{N_k}}a_{N_k}\ll \frac{\log(N_k)}{N_k}   
   \end{split}       
   \end{align}
yields, that the above equation $\eqref{gap1}$ holds for all numbers in $\tilde{I}_1:=[\mathfrak{a}_1,\mathfrak{a}_2]$, where $|\tilde{I}_1|=\frac{2\log(N_k)}{N_k\cdot a_{N_k}}$. 
Let us now consider the the set $\{a_n\}_{N_{k+1}<n\le 2N_{k+1}}$. By Corollary \ref{translated},
    \begin{equation}\label{gap2}
       G(\{\alpha_2  a_n\}_{N_{k+1}\le n\le 2N_{k+1}})\ll \frac{\log(N_{k+1})}{N_{k+1}}, 
    \end{equation}
    holds for an $\alpha_2$ in every interval $I_1$ of length $|I_1|=4/a_{N_{k+1}}$.
    Then, we can choose $I_1$ to be inside of $\tilde{I}_1$, since
   \begin{equation}\label{subsetinterval}
    |I_1|=\frac{4}{a_{N_{k+1}}}\le\frac{4}{r^{3\cdot4^{k}}a_{N_{k}}}
    \le\frac{\log(4^k)}{4^{k} a_N}=\frac{|\tilde{I}_1|}{2},
\end{equation}
for $k$ sufficiently large.
    In particular $\alpha_2\in I_1\subseteq \tilde{I}_1$ such that 
    $$
     G(\{\alpha_2  a_n\}_{n\le N_k})\ll \frac{\log(N_k)}{N_k}
    $$
    as well as
    \begin{equation}\label{gap1.5}
      G(\{\alpha_2  a_n\}_{N_{k+1}<n\le 2N_{k+1}})\ll \frac{\log(N_{k+1})}{N_{k+1}}. 
    \end{equation}
By the same argument as before, see equation \eqref{interval}, there exists an interval $I_2$ such that \eqref{gap1.5} holds for every number in $\tilde{I}_2=(\alpha_2-\tau_2,\alpha_2+\tau_2)$, $\tau_2=\frac{\log(N_{k+1})}{N_{k+1}\cdot a_{N_{k+1}}}$. By $\alpha_2\in I_1$ at least half of the interval $\tilde{I}_2$ is either inside of $I_1$ or covers it.
Next consider the set $\{a_n\}_{N_{k+2}< n\le 2N_{k+2}}$ for which we get a solution $\alpha_3$ in every interval $I_2$  of length $|I_2|=4/a_{N_{k+2}}$. By a simple induction argument, equation \eqref{subsetinterval} implies $|I_2|\le|\tilde{I}_2|/2$. Thus we find an $\alpha_3$ such that
\begin{equation}\label{gap4.5}
       G(\{\alpha_3  a_n\}_{N_{\tilde{k}}<n\le 2N_{\tilde{k}}})\ll \frac{\log(N_{\tilde{k}})}{N_{\tilde{k}}},
    \end{equation}
    with $\tilde{k}=k,k+1,k+2$.
This procedure can be repeated infinitely many times. Hence we get a sequence of intervals of exponentially decreasing length, that are nested inside each other.
Thus we find a real number $\alpha_{\infty},$ which is defined by
$$\bigcap^{\infty}_{n=1}I_n=\{\alpha_{\infty}\},$$
such that for all positive integers $N_j$, $j\in\N$ up to finitely many, we have
\begin{equation}\label{gapj}
       G\left(\{\alpha_{\infty}  a_n\}_{N_j<n\le 2N_j}\right)\ll \frac{\log(N_j)}{N_j}. 
    \end{equation}
Furthermore, since adding points can only decrease the gap, we have:
\begin{equation}\label{interindex}
 G\left(\{\alpha_{\infty}  a_n\}_{n\le 2N_j}\right)\leq G\left(\{\alpha_{\infty}  a_n\}_{N_j<n\le 2N_j}\right)\ll\frac{\log(N_j)}{N_j}\simeq\frac{\log(2N_j)}{2N_j}.   
\end{equation}
Substituting $2N_j$ by $M_j$ allows to consider a set of lacunary type starting at index $1$.\\   
For all the other positive integers, we get the result immediately from the following interpolation argument: For every $N\ge M_k$, there exists an $m\in\N$ such that
$M_m\le N\le M_{m+1}$, which yields $$\{\alpha  a_n\}_{n\le M_m}\subseteq\{\alpha  a_n\}_{n\le N}\subseteq\{\alpha a_n\}_{n\le M_{m+1}}.$$
Again, since adding more points can only decrease the maximal gap-size, we have
\begin{equation}
    G(\{\alpha  a_n\}_{n\le N})\ll \frac{\log(M_m)}{M_m}\simeq \frac{\log( 4^{m+1})}{ 4^{m+1}}=\frac{\log(M_{k+1})}{M_{k+1}}\le\frac{\log(N)}{N}.
\end{equation}

\end{proof}

\section{The metric problem in the case of the Lebesgue measure. Proof of Theorem \ref{main3}.}

In this section we investigate the maximal gap for almost all $\alpha\in[0,1]$, see Theorem \ref{main3}.
 The upper bound we obtain is worse than the bound in the existence result in Theorem \ref{main2}, but we have a stronger result with regard to $\alpha$. Theorem \ref{main3} is an improvement of the existing results in \cite{chow2023dispersion}. To prove the theorem, we will use compactly supported $C^{\infty}$-functions that are non-zero only in an open interval of length $\log(N)^{2+\varepsilon}/N$ centered around $N/\log(N)^{2+\varepsilon}$ the equidistant points. This means these functions locate points of the given set $\{\alpha a_n\}_{n\le N}$ in a neighbourhood of the just mentioned size.
 Some Fourier analysis, together with an application of Markov's inequality and the first Borel-Cantelli Lemma, yields that these functions iare positive for almost all $\alpha\in[0,1]$ and $N$ sufficiently large, which in particular means that the inequality $\eqref{introeq}$ is true for almost all $\alpha\in[0,1]$.

\subsection{Some prerequisites}\label{pre2}
For clarity of exposition, we start by first only considering the Lebesgue measure, which we denote as $\lambda$. Hence when writing down an integral, for the moment it is always understood to be taken with respect to the Lebesgue measure. We write $\mathcal{S}(\mathbb{R})$ for the Schwartz-space and $C^{\infty}_I$ for the space of arbitrarily often continuously differentiable functions with support on an interval $I$. We define the Fourier Transform $\mathcal{F}:\mathcal{S}(\R)\rightarrow \mathcal{S}(\R)$, that is in particular an automorphism, as follows: $\mathcal{F}f(x):=\int_{\R}f(\xi)e^{-2\pi ix\xi}d\xi$, see \cite{fourier} for details and basic properties. We remark $C_I^{\infty}\subset\mathcal{S}$ and that every $f\in\mathcal{S}$ is bounded. For $f\in\mathcal{S}$ we recall a version of the Poisson summation formula:
\begin{equation}\label{pois}
   \sum_{n\in\Z}f((\nu+n)T)=\frac{1}{T} \sum_{k\in\Z} \mathcal{F}f\left(\frac{k}{T}\right)e^{-2\pi ik\nu}, 
\end{equation}
see e.g. \cite{gröchenig2001foundations}. Next we define the $L^p([0,1])-$norm for $1\le p<\infty$ as: $||f||_{L^p}:=\left(\int_0^1|f|^p d\lambda\right)^\frac{1}{p}$. The functions  $\sqrt{2}\cos(2\pi kx)$, $\sqrt{2}\sin(2\pi kx)$ $k\in\Z$ are orthonormal in $L^2([0,1]):=\{f\ \textnormal{measurable}: ||f||_{L^2}<\infty\}$. We further recall a version of Markov's inequality: Let $f$ be measurable, then
\begin{equation}\label{Markov'slambda}
     \lambda(\{\alpha\in [0,1]\,:\,\,|f(\alpha)|\geq t\}) \leq \left(\frac{||f||_p}{t} \right)^p  
\end{equation}
Finally, we need a consequence of the first Borel Cantelli Lemma: Let $f_n$ measurable, $(X_n)_{n\in\N}\subset\R^+$ an increasing sequence and assume that 
\begin{equation}\label{bClambda}
    \sum_{n=1}^\infty \lambda(\{\alpha\in [0,1]: | f_n(\alpha) - X_n | > X_n/2\} ) < \infty.
\end{equation}
Then $f_n(\alpha)>0$ for almost all $\alpha\in [0,1]$ and all $n\in\N$ sufficiently large. These are all standard results, that can be found in e.g. \cite{durrett2019probability}.
\begin{rem}
    Markov's inequality and the first Borel Cantelli Lemma are still true for any other Borel-measure.
\end{rem}


\textbf{Proof of Theorem \ref{main3}:}
\begin{proof}
\textbf{Step 1: Some Fourier Analysis:}
Let $\varepsilon>0$ and $\{a_n\}_{n\le N}$ be a set of lacunary type with subset $\{\tilde{a}_n\}_{n\le K}$, where $K=N/l\log(N)^{1+2\varepsilon}$, $l=\max\{1,\lceil1/\log(r)\rceil+1/2\}$ and $\ta_n=a_{nl\log(N)^{1+2\varepsilon}}$.
   We remark that the set $\{\tilde{a}_n\}_{n\le K}$ is linearly independent in $[-K,K]\cap\Z$ for $N$ sufficiently large. This follows from 
   \begin{equation}\label{linindep}
        \sum_{j=1}^{K_0-1} 4l^{-1}K\tilde{a_j}< \tilde{a}_{K_0},
   \end{equation}
   for every $K_0\in[2,K]\cap\Z$, by the same calculation as in equations \eqref{estimate}.
   Without loss of generality let $l=1$, since everything can be directly adapted to the case $l\ne 1$. 
Let $f\in C_{[-1,1]}^{\infty}$ real and even, such that $\int_{\R}f(x)dx=1$. This in particular means, that $Ff$ is real and even, $\mathcal{F}f(0)=1$, and $|\mathcal{F}f|\le1$. We set $Q:=\log(N)^{1+2\varepsilon}$, $M:=\log(N)^{2+4\varepsilon}=Q^2$, $P:=\log(N)^{2+3\varepsilon}$, $R:=\log(N)^{\varepsilon}=M/P$ and define the following function, which counts points in intervals of length $M/N$:
$$\omega_{N,t}(\alpha):=\sum_{u\in\Z}\sum_{n=1}^{N/Q}f\left(\frac{\alpha \tilde{a}_n-t-u}{\frac{M}{N}}\right).$$
The sum over $u$ allows us to consider the rational part; the denominator dilates the function to count only points within an interval of length $N/M$, and $t$ translates the interval. Using Poisson summation formula \eqref{pois} we get:
$$\omega_{N,t}(\alpha)=\frac{M}{N}\sum_{k\in\Z}\sum_{n=1}^{N/Q}(\mathcal{F}f)\left(\frac{M}{N}k\right)e^{2\pi ik(\alpha \tilde{a}_n-t)}.$$
We now define:
\begin{align}
    \begin{split}
       \tilde{\omega}_{N,t}(\alpha)&:=\frac{M}{N}\sum_{k\in\Z}\sum_{n=1}^{N/Q}(\mathcal{F}f)\left(\frac{M}{N}k\right)e^{2\pi ik(\alpha \tilde{a}_n-t)}-Q\\\
&=\frac{M}{N}\sum_{k\in\Z\backslash\{0\}}\sum_{n=1}^{N/Q}(\mathcal{F}f)\left(\frac{M}{N}k\right)e^{2\pi ik(\alpha \tilde{a}_n-t)}\\\
    \end{split}
\end{align}
Here we used that $(\mathcal{F}f)(0)=0$ and $M/Q=Q$. 
We truncate the first sum and work instead with
\begin{align}
    \begin{split}
       \omega^{*}_{N,t}(\alpha)&:=\frac{M}{N}\sum_{0<|k|\le N/P}\sum_{n=1}^{N/Q}(\mathcal{F}f)\left(\frac{M}{N}k\right)e^{2\pi ik(\alpha \tilde{a}_n-t)},\\\
&=\frac{M}{N}\sum_{0<|k|\le N/P}\sum_{n=1}^{N/Q}(\mathcal{F}f)\left(\frac{M}{N}k\right)\cos(2\pi k(\alpha \tilde{a}_n-t)) 
    \end{split}
\end{align}
where by the rapid decay of $\mathcal{F}f$ (this follows by $\mathcal{F}$ being an automorphism and the definition of Schwartz-functions), we have
$$\tilde{\omega}_{N,t}(\alpha)=\omega^{*}_{N,t}(\alpha)+O(\log(N)^{-100}).$$
Define $p_k:=(\mathcal{F}f)\left(\frac{M}{N}k\right)$, $\tilde{p}_k:=p_k\cos(2\pi kt)$ and $p^*_k:=p_k\sin(2\pi kt)$. In particular $|p_k|,|\tilde{p}_k|,|p^*_k|\le1$. 
The next step is to estimate the following integral, which will be useful later, when using the first Borel-Cantelli-Lemma.

\begin{align}\label{lebesguemeasure}
\begin{split}
   &\ \ \ \ \ ||e^{\frac{1}{10R}\omega^*_{N,t}}||_{L^1}=\int_0^1\exp\left(\frac{1}{10R}\omega^*_{N,t}(\alpha)\right)d\lambda(\alpha)\\\
&=\int_0^1\prod_{n=1}^{\frac{N}{Q}}\exp
\left(\frac{1}{10R}\frac{M}{N}\sum_{0<|k|\le\frac{N}{P}}p_k\Bigl(\cos(2\pi k\tilde{a}_n\alpha)\cos(2\pi kt)+\sin(2\pi k\tilde{a}_n\alpha)\sin(2\pi kt)\Bigl)\right)\\\
&\le\int_0^1\prod_{n=1}^{\frac{N}{Q}}\Biggl(1+\frac{1}{10R}\frac{M}{N}\sum_{0<|k|\le\frac{N}{P}}\Bigl(\tilde{p}_k\cos(2\pi k\tilde{a}_n\alpha)+p^*_k\sin(2\pi k\tilde{a}_n\alpha)\Bigr)\\\
&+\biggl(\frac{1}{10R}\frac{M}{N}\sum_{0<|k|\le\frac{N}{P}}\bigl(\tilde{p}_k\cos(2\pi k\tilde{a}_n\alpha)+p^*_k\sin(2\pi k\tilde{a}_n\alpha)\bigr)\biggr)^2\Biggr)d\lambda(\alpha)\\\
&\le\prod_{n=1}^{\frac{N}{Q}}\Biggl(\biggl(1+\frac{1}{100R^2}\biggl(\frac{M}{N}\biggr)^2\sum_{0<|k|\le\frac{N}{P}}\max\{\tilde{p}_k,p^*_k\}^2\Biggr)\\\
&\le\exp\left(\sum_{n=1}^{\frac{N}{Q}}\frac{1}{100R^2}\frac{M^2}{N^2}\sum_{0<|k|\le\frac{N}{P}}\max\{\tilde{p}_k,p^*_k\}^2\right)\le\exp\left(\frac{Q}{50R}\right). 
\end{split}
\end{align}
Here we could remove the absolute values in the first line, since $\omega^*_{N,c}$ is real, hence $\exp\left(\frac{1}{10R}\omega^*_{N,t}\right)$ is positive.
The estimate from the second to the third line is an application of the Taylor-series of exponential functions; in detail $|x|\le 1$ implies $e^x\le 1+x+x^2$, and
$$\left|\frac{1}{10R}\frac{M}{N}\sum_{0<|k|\le\frac{N}{P}}p_k\cos(2\pi k(\alpha \tilde{a}_n-t))\right|\le\frac{1}{10R}\frac{M}{N}2\frac{N}{P}=\frac{1}{5}<1.$$
By the linear independence of $\{\tilde{a}_n\}_{n\le K}$ in $[-K,K]=[-N/M,N/M]$, the $L^2[0,1]$-orthonormality property and basic trigonometric identities, the integrals over products of cosines and sines with other cosines and sines with much larger frequencies vanish, which yields the penultimate line. The last line follows from the Bernoulli-inequality and trivial estimates. We note that Chow and Technau \cite{chow2023dispersion} worked with $L^p$ moments for large $p=p(N)$, following the work of Rudnick and Zaharescu \cite{rz}. The use of exponential integrals in our argument (essentially a moment-generating function in the language of probability theory), together with the application of Taylor's formula, effectively reduces the whole calculation to an $L^2$ argument, which avoids the very delicate combinatorial problems which arise in \cite{chow2023dispersion,rz}.\\

\textbf{Step 2: Some Measure Theory:}
The calculation above has been carried out in order to allow an application of Markov's inequality and utilize the first Borel-Cantelli-Lemma. We have
\begin{align}\label{measure}
\begin{split}
    &\ \ \ \ \lambda(\{\alpha\in [0,1]\,:\,\omega_{N,t}(\alpha)-Q\geq \frac{Q}{2}\})\\\
    &\le\lambda(\{\alpha\in [0,1]\,:\,\omega^{*}_{N,t}(\alpha)\geq \frac{Q}{4}\})\\\ 
    &=\lambda(\{\alpha\in [0,1]\,:\,\exp(\frac{1}{10R}\omega^{*}_{N,t}(\alpha))\geq \exp(\frac{Q}{40R})\})\\\
    &\le \frac{||e^{\frac{1}{10R}\omega^*_{N,t}}||_{L^1}}{e^{\frac{Q}{40R}}}\le e^{-\frac{Q}{200R}}=e^{-\frac{1}{200}\log(N)^{1+\varepsilon}}
  \end{split}
\end{align}
Following this, we consider \(N/\log(N)^{2+4\varepsilon}\) different translation variables \(t_j\) such that the center of \(\text{supp}(w^*_{N,t_j})\) is \(j \log(N)^{2+4\varepsilon}/N\), for \(j = 1, \ldots, N/\log(N)^{2+4\varepsilon}\), which are the centers of open intervals of length \(N/\log(N)^{2+4\varepsilon}\). Note, that these intervals are disjoint.
Since
$$\sum_{N\in\N}\sum_{j=1}^{\frac{N}{\log(N)^{2+4\varepsilon}}}\lambda(\{\alpha\in [0,1]\,:\,\omega_{N,t_j}(\alpha)-Q\geq \frac{Q}{40}\})\le\sum_{N\in\N}\frac{Ne^{-\frac{1}{200}\log(N)^{1+\varepsilon}}}{\log(N)^{2+4\varepsilon}}<\infty,$$
the first Borel Cantelli Lemma yields that for all sufficiently large $N$ and for all \\
$j=1,...,N/\log(N)^{2+4\varepsilon}$
$$\omega_{N,t_j}(\alpha)>0,$$
for almost all $\alpha\in[0,1]$. In other words, for sufficiently large $N$ every interval of length $\log(N)^{2+4\varepsilon}/N$ centered at of the points $t_j,$ $j=1,...,N/\log(N)^{2+4\varepsilon}$, contains at least one point of $\{\tilde{a}_n\}_{n\le N}$. Hence
\begin{equation}
    G(\{\alpha  \tilde{a}_n\}_{n\le K})\leq \frac{\log(N)^{2+4\varepsilon}}{N},
\end{equation}
for $N$ sufficiently large and for almost all $\alpha\in[0,1]$. As before, adding points can only decrease the maximal gap, which implies:
\begin{equation}\label{endresult}
    G(\{\alpha  a_n\}_{n\le N})\leq \frac{\log(N)^{2+4\varepsilon}}{N},
\end{equation}
for $N$ sufficiently large and for almost all $\alpha\in[0,1]$. Since $\varepsilon>0$ was arbitrary, this proves the theorem.
\end{proof}

\section{The metric problem in the case of general measures. Proof of Theorem \ref{main35}.} 

Again we write $\lambda$ for the Lebesgue measure. For a Borel-set $\mathcal{B}\subseteq\R$ we define $\mathcal{M}(\mathcal{B})$ as the set of Borel measures supported on a compact subset of $\mathcal{B}$. In the following, when writing a measure $\mu$, it will always be in $\mathcal{M}([0,1])$. We define the Fourier Transform of a measure as $$(\mathcal{F}\mu)(x):=\int_0^1e^{-2\pi i\xi x}d\mu(\xi).$$
The $L^p([0,1],\mu)$-norm is denoted as $||.||_{L^p(\mu)}$ and defined as before, we just exchange the Lebesgue measure $\lambda$ with a general measure $\mu$.\\


\textbf{Proof of Theorem \ref{main35}}
\begin{proof}
The idea is very similar to the proof of Theorem \ref{main3}. 
Let $\varepsilon>0$ and $\{a_n\}_{N<n\le 2N}$ be a set of lacunary type with subset $\{\tilde{a}_n\}_{K<n\le 2K}$, $\tilde{a}_n=a_{nl\log(N)}$, where $K=N/l\log(N)^{1+2\varepsilon}$
   with $r>1$ and $l=\max\{1,\lceil1/\log(r)\rceil$+1/2\}
Again without loss of generality let $l=1$.
Let $M,Q,P,R,p,\tilde{p},p^*$ be the same as before. We define $\omega_{N,t}$, $\tilde{\omega}_{N,t}$ and $\omega^*_{N,t}$ identically, except for the sequence inside, which now must be $\{\tilde{a}_n\}_{K<n\le 2K}$. First recall the Fourier-series expansion of 
$$e^{\frac{1}{10R}\omega^*_{N,t}(\alpha)}:=\frac{q_0}{2}+\sum_{m\in\N}q_m\cos(2\pi m\alpha)+\sum_{m\in\N}\tilde{q}_m\sin(2\pi m\alpha),$$
where
$$ q_m=\int_0^1e^{\frac{1}{10R}\omega^*_{N,t}(x)}\cos(2\pi mx)d\lambda(x),\quad \tilde{q}_m=\int_0^1e^{\frac{1}{10R}\omega^*_{N,t}(x)}\sin(2\pi mx)d\lambda(x).$$
With this in hand we get:
\begin{align}
    \begin{split}
        &||e^{\frac{1}{10R}\omega^*_{N,t}}||_{L^1(\mu)}=\left|\int_0^1e^{\frac{1}{10R}\omega^*_{N,t}(\alpha)}d\mu(\alpha)\right|\\\
        \ll_{\mu}&\left|\frac{q_0}{2}\right|+\sum_{m\in\N}\left|q_m\right|\left|\int_0^1\cos(2\pi m\alpha)d\mu(\alpha)\right|+\sum_{m\in\N}\left|\tilde{q}_m\right|\left|\int_0^1\sin(2\pi m\alpha)d\mu(\alpha)\right|\\\
        \le&\left|\frac{q_0}{2}\right|+\sum_{m\in\N}\left|q_m\right|\left(1+|m|\right)^{-\Upsilon}+\sum_{m\in\N}\left|\tilde{q}_m\right|\left(1+|m|\right)^{-\Upsilon}
    \end{split}
\end{align}
Here we used the assumption on the Fourier decay of $\mu$, which was part of the statement of Theorem \ref{main35}.
The first term can be estimated by the literally same procedure as in equation \eqref{lebesguemeasure}.
Our conclusion stems from the same reasoning for both sums, hence it is sufficient to only consider $\sum_{m\in\N}q_m$.
We split it into two parts, according to weather $|m|\le a_N$ or $|m|>a_N$. In the first partial sum we estimate in a very similar way as before. 
\begin{align}
    \begin{split}
        &\sum_{|m|\le a_N}\int_0^1e^{\frac{1}{10R}\omega^*_{N,t}(\alpha)}\cos(2\pi m\alpha)d\lambda(x)\left(1+|m|\right)^{-\Upsilon}\\\
        \ll_{\mu}&\sum_{|m|\le a_N}\int_0^1\prod_{n=K+1}^{\frac{2N}{Q}}\Biggl[1+\frac{1}{10R}\frac{M}{N}\sum_{0<|k|\le\frac{N}{P}}\Bigl(\tilde{p}_k\cos(2\pi k\tilde{a}_n\alpha)+p_k^*\sin(2\pi k\tilde{a}_n\alpha)\Bigr)\\\
        &+\biggl(\frac{1}{10R}\frac{M}{N}\sum_{0<|k|\le\frac{N}{P}}\Bigl(\tilde{p}_k\cos(2\pi k\tilde{a}_n\alpha)+p_k^*\sin(2\pi k\tilde{a}_n\alpha)\Bigr)\biggr)^2\Biggr]\cos(2\pi m\alpha)d\lambda(\alpha)\\\
        \le&\exp\left(\sum_{n=1}^{\frac{N}{Q}}\frac{1}{100}\frac{M^2}{N^2}\sum_{0<|k|\le\frac{N}{P}}p_k^2\right)\le\exp\left(\frac{1}{50}Q\right),
    \end{split}
\end{align}
Since the smallest value $\tilde{a}_n$ can attain is $\tilde{a}_{K+1}=a_{N+\log(N)}>a_N$, and is thus larger than the largest value $m$ can attain, and since we have linear independence by the same argument as in equation \eqref{linindep}, the functions
$\cos(2\pi m\alpha)$, with $m\le a_N$ have a too small frequency for orthogonality in $L^2([0,1])$ to fail. This means we can proceed as before, see equation $\eqref{lebesguemeasure}$, to get a similar estimate. \\  
In the second partial sum we will make use of the rapid decay of $(\mathcal{F}\mu)(-m)$:
\begin{align}
    \begin{split}
        &\sum_{|m|>a_N}\int_0^1e^{\frac{1}{10R}\omega^*_{N,t}(\alpha)}\cos(2\pi m\alpha)d\lambda(x)\left(1+|m|\right)^{-\Upsilon}\\\
\ll&\sum_{|m|>a_N}\int_0^1\prod_{n=K+1}^{\frac{2N}{Q}}\Biggl[1+\frac{1}{10R}\frac{M}{N}\sum_{0<|k|\le\frac{N}{P}}\Bigl(\tilde{p}_k\cos(2\pi k\tilde{a}_n\alpha)+p_k^*\sin(2\pi k\tilde{a}_n\alpha)\Bigr)\\\
&+\biggl(\frac{1}{10R}\frac{M}{N}\sum_{0<|k|\le\frac{N}{P}}\bigl(\tilde{p}_k\cos(2\pi k\tilde{a}_n\alpha)+p_k^*\sin(2\pi k\tilde{a}_n\alpha)\bigr)\biggr)^2\Biggr]\cos(2\pi m\alpha)d\lambda(\alpha)a_N^{-\Upsilon}
    \end{split}
\end{align}
Multiplying out the above product we get at least $7^{N/Q}$ terms, all of which have factors either $1$, $$\frac{1}{10R}\frac{M}{N}\sum_{0<|k|\le\frac{N}{P}}\tilde{p}_k\cos(2\pi k\tilde{a}_n\alpha)\quad\textnormal{or}\quad\frac{1}{10R}\frac{M}{N}\sum_{0<|k|\le\frac{N}{P}}p_k^*\sin(2\pi k\tilde{a}_n\alpha),$$ 
for $n\in\{K+1,...,2N/Q\}$.
Upon multiplying out also the sums in $k$, we obtain at most $7^{N/Q}(2N/P)^{2N/Q}$ terms. Consequently, there are at most that many terms which are not orthogonal to some $m>a_N$. Since $M/(NR)< 1$ and $\tilde{p}_k,p_k^*\le1$, we can estimate the above by
\begin{align}
    \begin{split}
     &7^{N/Q}(2N/P)^{2N/Q}a_N^{-\Upsilon}\\\
      \le&\exp\left(\frac{3N}{\log(N)^{1+2\varepsilon}}(\log(N)+\log(2)-\log(P))-\log(r)N\right) \\\
      \le&\exp\left(\frac{3N}{\log(N)^{2\varepsilon}}-\frac{4N}{\log(N)^{2\varepsilon}}\right)<1\le \exp\left(\frac{Q}{50R}\right),
    \end{split}
\end{align}
for $N$ large enough.
Combining the above yields
$$||\omega^*_{N,t}(\alpha)||_{L^1(\mu)}\ll_{\mu} \exp\left(\frac{Q}{50R}\right),$$
which is the same bound as before (up to constants). The rest of the proof is identical to the proof of Theorem \ref{main3}. 
\end{proof}

\section{Application to Littlewood's Conjecture in multiplicative Diophantine approximation} 
\label{sec:little}

In this section we assume some familiarity with the notions and basic results of Diophantine approximation theory. For readers not familiar with the subject, we refer to the monographs of Bugeaud \cite{bug}, Niven \cite{niven} or Schmidt \cite{schmidt}.\\

Littlewood's conjecture is one of the central open problems in Diophantine approximation; it asserts that for all $\alpha,\beta \in  \mathbb{R}$,
$$
\liminf_{n \to \infty} n \|n \alpha\| \cdot \|n \beta\| =0. 
$$
There has been some progress towards this conjecture, in particular by the work of Einsiedler, Katok and Lindenstrauss \cite{ekl}. However, as a whole, the conjecture is still widely open. It should be noted that the problem is not only of central importance in Diophantine approximation, but also for other mathematical areas, in particular in the context of measure rigidity in the theory of homogeneous dynamics (see for example \cite{marg}). Note that Littlewood's conjecture is a problem about fixed pairs $(\alpha,\beta)$ of reals. Gallagher studied the problem from a metric perspective, and obtained a convergence/divergence criterion for whether there exist finitely/infinitely many $n$ such that 
$$ 
\|\alpha n\| \cdot \|\beta n\| < \psi(n)
$$
for almost all $(\alpha,\beta) \in [0,1]^2$, for a given approximation function $\psi$. Much work has been devoted to replace the metric setting of Gallagher's theorem (the underlying measure is the two-dimensional Lebesgue measure) by ``finer'' fractal measures. This is particularly relevant, since in all potential counterexamples to Littlewood's conjecture both $\alpha$ and $\beta$ would have to be badly approximable numbers, and the set of badly approximable numbers is a Lebesgue null set. Our theorem addresses the inhomogeneous case, where the ``homogeneous'' distances $\|\alpha n\|$ and $\|\beta n\|$ are replaced by the inhomogeneous distances $||\alpha n-\eta||$ and $||\beta n-\zeta||$ with some fixed shift variables $\eta$ and $\zeta$. The study of such inhomogeneous problems in Diophantine approximation has been very popular recently; we refer to \cite{bv,chowzaf,chow,ct2,hs,pvzz,ram}.\\

We define the set of non-Liouville numbers $\mathcal{L}:=\{x \in\R:\Lambda(\alpha)<\infty\}$, where $\Lambda(\alpha)=\sup\left\{\log(\rho_k(\alpha))/k:k\ge 1\right\}$, with $\rho_k(\alpha)$ being the continuants of $\alpha$; for details on definitions see e.g. \cite{numbana}. We remark that $\mathcal{L}$ includes almost all real numbers, because $\Lambda(\alpha)$ equals the L\'{e}vi-constant $\pi^2/12\log(2)$ for almost all $\alpha\in\R$, see Khintchine \cite{khintchine1926metrischen} and L\'{e}vi \cite{Lévy1936}. Finally we define the set $\textnormal{\textbf{Bad}}:=\{\alpha\in\R:\inf\{n||n\alpha||>0\}$, which is called the set of badly approximable numbers.

\begin{thm}\label{hilf}
Let $\varepsilon>0$, $\mu\in\mathcal{M}([0,1])$, $\eta\in\R$, let $(a_n)_{n\in\N}$ be a lacunary sequence of positive integers, and assume that its Fourier Transform decays as $\left|(\mathcal{F}\mu)(x)\right|\ll(1+|x|)^{-\Upsilon}$ for some $\Upsilon>0$. Then
\begin{equation}\label{hilf1}
    ||\alpha a_n-\eta||\leq \frac{\log(n)^{2+\varepsilon}}{n}
\end{equation}
has infinitely many solutions $n\in\N$ for $\mu$-almost all $\alpha\in[0,1]$.
\end{thm}
\begin{proof}
Theorem \ref{main35} implies that for every $\eta\in\R$
there exists a number $n\in\N$ with $N<n\le 2N$ such that
$$||\alpha a_n-\eta||\leq\frac{\log(N)^{2+\varepsilon}}{N}\le \frac{\log(n)^{2+\varepsilon}}{n},$$
for every $N\in\N$ (except for finitely many $N$, which do not matter) and almost all $\alpha\in[0,1]$.
\end{proof} 
\begin{thm}[Chow-Zafeiropoulos \cite{chowzaf}, 2021]\label{chowza}
  Let $\beta\in\mathcal{L}$ and $\zeta\in\R$. Then there exists a lacunary sequence $(a_n)_{n\in\N}$ of positive integers such that 
  \begin{equation}\label{zaf1}
     8^n<a_n\le4^{6\Lambda(\beta)n} 
  \end{equation}
  and
  \begin{equation}\label{zaf2}
      a_n||\beta a_n-\zeta||\le 8
  \end{equation}
  for all $n\in\N$.
\end{thm}

\begin{thm}\label{chthm}
    Let $\varepsilon>0$, $\beta\in\mathcal{L}$ and $\zeta\in\R$. Define 
   $\mathfrak{A}$ as the set of $\alpha\in[0,1]$ for which there exists $\eta\in\R$, such that $$n||\alpha n-\eta||\cdot||\beta n-\zeta||\leq \frac{\log(\log(n))^{2+\varepsilon}}{\log(n)}$$ has only finitely many solutions $n\in\N$. Let $\mu\in\mathcal{M}([0,1])$ with $(\mathcal{F}\mu)(x)\ll(1+|x|)^{-\Upsilon},$ with $\Upsilon>0$. Then 
  $$\mu(\mathfrak{A})=0.$$
\end{thm}
This is a special case of a more general Theorem of Chow and Technau \cite{chow2023dispersion}. For the convenience of the reader we include a proof.
\begin{proof}
    Since Theorem \ref{hilf} holds for every lacunary sequence, we choose the one we get from the conclusion of Theorem \ref{chowza}, and combine the statements \eqref{hilf1} and \eqref{zaf2}: For $\mu-$almost all $\alpha\in[0,1]$, there exists infinitely many $n$ such that
   $$ a_n||\beta a_n-\zeta||\cdot||\alpha a_n-\eta||\ll\frac{\log(n)^{2+\varepsilon}}{n},$$
   for all $\eta\in\R$.
  By the second inequality of \eqref{zaf1} we have $\log(a_n)\ll_{\beta} n$, which implies that
  $$\frac{\log(n)^{2+\varepsilon}}{n}\ll_{\beta}\frac{\log(\log(a_n))^{2+\varepsilon}}{\log(a_n)}.$$ 
  Hence for $\mu$-almost all $\alpha\in[0,1]$
  $$ a_n||\alpha a_n-\eta||\cdot||\beta a_n-\zeta||\ll_{\beta}\frac{\log(\log(a_n))^{2+\varepsilon}}{\log(a_n)},$$
  for all $\eta\in\R$ and for infinitely many $n$.
\end{proof}

\begin{proof}[Proof of Theorem \ref{little}]
    Again, we follow an argument of Chow and Technau \cite{chow2023dispersion}. We use a measure $\mu$ Kaufman constructed in \cite{Kaufman_1980}, that satisfies the following properties:
    \begin{itemize}
        \item $\mu\in\mathcal{M}([0,1]\cap\textnormal{\textbf{Bad}}),$
        \item For every interval $I\subseteq[0,1]$, we have $\mu(I)\ll_{\varrho}\lambda(I)^{\varrho},$
        \item For $x\in\R$: $\left|(\mathcal{F}\mu)(x)\right|\ll (1+|x|)^{-\left(10^{-5}\right)}$.
    \end{itemize}
Let $\mathfrak{A}$ be as before. Then Theorem \ref{chthm} yields $\mu(\mathfrak{A})=0$. Let $\mathfrak{G}:=\textnormal{\textbf{Bad}}\cap([0,1]\backslash\mathfrak{A}),$ then $\mu(\mathfrak{G})=1$. Finally we use the mass distribution theorem, see e.g. \cite{assani2013ergodic} or \cite{falconer2004fractal}, which yields in our case $\textnormal{dim}_H(\mathfrak{G})\ge \varrho$. The constant $\varrho$ can be chosen arbitrarily close to 1, hence dim$_H(\mathfrak{G})=1$.
\end{proof}

\textbf{Acknowledgments}
The author was supported by the Austrian Science Fund (FWF), project P-35322. The author expresses gratitude to Christoph Aistleitner, Athanasios Sourmelidis, and Bence Borda for their valuable discussions. Additionally, the author appreciates the insightful exchanges with Niclas Technau and Sam Chow.

\printbibliography{}






\end{document}